\documentclass[12pt]{article}
\usepackage[cm]{fullpage}

\usepackage{amscd}
\usepackage{amsmath}
\usepackage{amsfonts}
\usepackage{pb-diagram}

\newcommand{\tr}[2]{\textrm{tr}_{#1} \, {#2}}
\newcommand{\ve}{\varepsilon}

\newcommand{\omegas}{\omega_{\textrm{sing}}}
\newcommand{\oloc}{\omega_{\textrm{loc}}}

\newcommand{\ofs}{\omega_{\textrm{FS}}}

\newcommand{\dbar}{\overline{\partial}}

\newcommand{\ddt}[1]{\frac{\partial #1}{\partial t}}

\newcommand{\ov}[1]{\overline{#1}}

\newcommand{\op}{\omega_{\textrm{prod}}}

\newcommand{\ddbar}{\frac{\sqrt{-1}}{2\pi} \partial\dbar}

\begin{document}
\newcounter{remark}
\newcounter{theor}
\setcounter{remark}{0}
\setcounter{theor}{1}
\newtheorem{claim}{Claim}
\newtheorem{theorem}{Theorem}[section]
\newtheorem{proposition}{Proposition}[section]
\newtheorem{conjecture}{Conjecture}[section]
\newtheorem{question}{Question}[section]
\newtheorem{lemma}{Lemma}[section]
\newtheorem{definition}{Definition}[theor]
\newtheorem{corollary}{Corollary}[section]
\newenvironment{proof}[1][Proof]{\begin{trivlist}
\item[\hskip \labelsep {\itshape #1}]}{\hfill$\square$\medskip\end{trivlist}}
\newenvironment{remark}[1][Remark]{\addtocounter{remark}{1} \begin{trivlist}
\item[\hskip
\labelsep {\bfseries #1  \thesection.\theremark}]}{\end{trivlist}}
\setlength{\arraycolsep}{2pt}

\centerline{ \bf \large THE K\"AHLER-RICCI FLOW}

\smallskip

\centerline{\bf \large ON PROJECTIVE BUNDLES\footnote{The first-named author was supported in part by an NSF CAREER grant 
  DMS-08-47524 and a Sloan Fellowship;  the second-named author  by the NSF grant DMS-09-04223; the third-named author  by the NSF grants DMS-08-48193, DMS-11-05373 and a Sloan Fellowship.} }
  \bigskip

\centerline{ Jian Song$^*$, G\'abor Sz\'ekelyhidi$^{**}$ and Ben Weinkove$^{***}$}

\bigskip

\abstract{We study the behaviour of the K\"ahler-Ricci flow on
projective bundles. We show that if the initial metric is in a suitable
K\"ahler class, then the fibers collapse in finite time and the metrics
converge subsequentially in the Gromov-Hausdorff sense to a metric on the base.}

\bigskip

\section{Introduction}

Let $X$ be a projective bundle over a smooth projective variety $B$.
This means that $X=\mathbb{P}(E)$, where $\pi: E\to B$ is a holomorphic vector bundle of rank $r$, say. We study the behavior of the solution $\omega=\omega(t)$ of the
K\"ahler-Ricci flow on $X$ starting at a K\"ahler metric $\omega_0$:
\begin{equation} \label{krf}
\ddt{} \omega = - \textrm{Ric}(\omega), \qquad \omega|_{t=0} = \omega_0.
\end{equation}

The solution $\omega(t)$ develops a singularity after a finite time.  Indeed from \cite{TZ}
a maximal smooth solution to (\ref{krf}) exists on $[0,T)$ where $T>0$ is given by
\begin{equation}
T = \sup \{ t >0 \ | \ [\omega_0] + t c_1(K_X) >0 \}.
\end{equation}
$T$ is finite since $F \cdot c_1(-K_X)^{r-1} >0$ for every fiber $F$.
For $t \in [0,T)$, $\omega(t)$ lies in the class $[\omega_0] + t c_1(K_X)$.
We will assume that 
the limiting K\"ahler class $[\omega_0] + T c_1(K_X)$ satisfies
\begin{equation} \label{condition}
[\omega_0] + T c_1(K_X) = [\pi^* \omega_{B}],
\end{equation}
for some K\"ahler metric $\omega_{B}$ on $B$.

Our main result shows that a sequence of metrics along the flow
converges subsequentially to a metric on $B$ in the Gromov-Hausdorff sense as $t\to T$. 
\begin{theorem}\label{thm:main}
	There exists a sequence of times $t_i\to T$ and a distance
	function $d_B$ on $B$ (which is uniformly equivalent to the
	distance induced by $\omega_B$), such that $(X,\omega(t_i))$
	converges to $(B,d_B)$ in the Gromov-Hausdorff sense. 
\end{theorem}

This implies that the fibers of $X$ collapse.
In order to prove this theorem we establish (Lemmas \ref{lem:bdd} and \ref{lem:diam}) the following estimates for a uniform constant $C$ and for all $t \in [0,T)$:

\begin{enumerate}
\item[(i)] $\omega(t) \le C \omega_0$, and
\item[(ii)] $\textrm{diam}_{\omega(t)} F \le C(T-t)^{1/3}$, for every fiber $F$.
\end{enumerate}

That is, the metrics $\omega(t)$ are uniformly bounded from above along the flow and the diameters of the fibers tend to zero as $t \rightarrow T$.

To illustrate how Theorem \ref{thm:main} ties in with the existing literature on the K\"ahler-Ricci flow we discuss the example of $X = \mathbb{P}^2$ blown up at one point,  which is a $\mathbb{P}^1$ bundle over $\mathbb{P}^1$.  As $t \rightarrow T$ the behavior of $\omega(t)$ in the sense of Gromov-Hausdorff depends on the point at which the K\"ahler classes $[\omega(t)]$ hit the boundary of the K\"ahler cone.
Write $\alpha = [\omega_0] + T c_1(K_X)$.  Then one of the following holds:
\pagebreak[3]
\begin{enumerate}
\item[(1)] $\alpha=0$.
\item[(2)] $\alpha$ is the pull-back of a K\"ahler class from the base $\mathbb{P}^1$ (the setting of this paper).  
\item[(3)] $\alpha$ is the pull-back of a K\"ahler class from $\mathbb{P}^2$ via the blow-down map $p : X \rightarrow \mathbb{P}^2$.
\end{enumerate}
In each case we have a map $f: X \rightarrow M$ to a manifold $M$ (of dimension 0, 1 and 2 respectively) and $\alpha = f^* \beta$ for $\beta$ a K\"ahler class on $M$.    Feldman-Ilmanen-Knopf \cite{FIK} made a number of conjectures about the behavior of (\ref{krf}) which they established for self-similar solutions.  In particular they conjectured that in each of these three cases the flow should converge in the Gromov-Hausdorff sense to a metric on $M$.  We now briefly describe some progress on these conjectures.

In case (1) it  is an immediate consequence of a result of Perelman \cite{P} (see \cite{SeT}) that $(X, \omega(t))$ converges in the Gromov-Hausdorff sense to a point.  When the initial metric is invariant under a $U(2)$ symmetry it was shown in \cite{SW1} that:  in case (2), $(X, \omega(t))$ converges in the Gromov-Hausdorff sense to the base $\mathbb{P}^1$ with the Fubini-Study metric;  and in case (3), $(X, \omega(t))$ converges in the Gromov-Hausdorff sense to $(\mathbb{P}^2, d)$ where $d$ is a metric on $\mathbb{P}^2$ inducing the usual topology.  In \cite{SW2} the same behavior for (3) was established without assuming symmetry of the initial data.  Thus the case of (2) without symmetry assumption was still open.

Although we have limited our discussion here to the case of $\mathbb{P}^2$ blown up at one point, the results of \cite{SW1}  and especially \cite{SW2} apply to a much larger class of manifolds.  Moreover,  these results fit into a general program of understanding the K\"ahler-Ricci flow on algebraic varieties and the analytic minimal model program \cite{Ts, SoT3, T} (see also \cite{EGZ, SoT1, SoT2, TZ, Zha1, Zha2}).

Our Theorem \ref{thm:main} is concerned with the setting of (2)  for a general projective bundle over an algebraic variety $B$.  Even in the case of a product $B \times \mathbb{P}^{r-1}$ with arbitrary initial metric $\omega_0$, the result of Theorem \ref{thm:main} is new.  

 Of course, the \cite{FIK} conjectures predict that the convergence should occur without taking subsequences.  A difficulty in  (2) compared to the setting  (3) is that the metric $\omega(t)$ becomes singular at \emph{every} point of the manifold $X$.  

A further interesting and difficult question is to analyze the singularity at time $T$ by a rescaling procedure (see the discussion in \cite{FIK}).  In case (1) above, rescaling so that $X$ has fixed volume, one obtains a compact K\"ahler-Ricci soliton \cite{WZ, Zhu, TZhu}. For more general manifolds with $c_1(X)>0$ this is related to a question of Yau \cite{Y2} regarding stability in the sense of geometric invariant theory (see for example \cite{CW, D, T1, PS, PSSW, Sz, T1, To}).  In case (2), rescaling so that the fibers have fixed volume, one would expect to obtain a product $\mathbb{P}^1 \times \mathbb{C}$ (see the recent preprint \cite{Fo} in the case of $U(2)$ symmetry).  In case (3) the appropriate rescaling should yield the shrinking soliton constructed in \cite{FIK}. 

In Section \ref{sec:estimates}, we establish the key estimates (i) and (ii)  mentioned above.  
 In Section~\ref{sec:mainproof} we 
prove Theorem~\ref{thm:main}.

\section{The main estimates}\label{sec:estimates}

In this section we establish the main estimates.  Assume $X$ has complex dimension $n$ and  $X = \mathbb{P}(E)$ where $E$ is a rank $r$ holomorphic vector bundle over $B$.  We will often write $g$ for the K\"ahler metric with K\"ahler form $\omega$.  We use $C$, $C'$ or $C_i$ to  denote a uniform constant, which may differ from line to line.

First, we have a lower bound for $\omega(t)$ 
from the parabolic Schwarz Lemma, as in  \cite{SoT1}.  

\begin{lemma} \label{schwarz}
There exists a uniform constant $c>0$ such that
\begin{equation}
\omega \ge c \pi^* \omega_{B}.
\end{equation}
\end{lemma}
\begin{proof}  This estimate is well-known to hold.   Indeed the argument is almost identical to the proof of Lemma 2.2 in \cite{SW2} (see also \cite{TZ}).  We provide a brief sketch for the reader's convenience.  First, we reformulate (\ref{krf}) as a parabolic complex Monge-Amp\`ere equation.   Define a family of reference metrics $\hat{\omega}_t \in [\omega(t)]$ by
\begin{equation} \label{reference}
\hat{\omega}_t = \frac{1}{T} ((T-t) \omega_0 + t \pi^* \omega_{B}),
\end{equation}
and let $\Omega$ be the unique volume form on $X$ satisfying
\begin{equation}
\ddbar \log \Omega = \ddt{} \hat{\omega}_t = \frac{1}{T} ( \pi^* \omega_{B} - \omega_0) \in c_1(K_X), \quad \int_X \Omega =1.
\end{equation}
 If $\varphi=\varphi(t)$ solves the parabolic complex Monge-Amp\`ere equation 
\begin{equation} \label{pcma}
\ddt{\varphi} = \log \frac{(\hat{\omega}_t + \ddbar \varphi)^n}{\Omega}, \qquad \hat{\omega}_t + \ddbar \varphi >0,  \qquad \varphi|_{t=0} = 0,
\end{equation}
then $\omega = \hat{\omega}_t + \ddbar \varphi$ solves (\ref{krf}).  Conversely, given the solution $\omega=\omega(t)$ of (\ref{krf}) on $[0,T)$, we can extract the unique solution $\varphi=\varphi(t)$ of (\ref{pcma}) for $t \in [0,T)$.

Since $c_0 (T-t)^n \Omega \le \hat{\omega}_t^n \le C_0 \Omega$ for uniform positive constants $c_0$  and  $C_0$, an elementary maximum principle argument shows that $\varphi$ is uniformly bounded (see \cite{TZ} or Lemma 2.1 of \cite{SW2}).  We then compute using the parabolic Schwarz lemma computation of \cite{SoT1} that 
\begin{equation}
\left( \ddt{} - \Delta \right) \log  \tr{\omega}{(\pi^* \omega_B)}  \le C \tr{\omega}{(\pi^* \omega_B)},
\end{equation}
for a uniform constant $C$.  On the other hand, for a uniform constant $c'>0$,
\begin{equation}
 \Delta \varphi =  \tr{\omega}{(\omega - \hat{\omega}_t)} \le n - c' \tr{\omega}{(\pi^* \omega_B)}.
\end{equation}
The result follows by applying the maximum principle to the quantity
\begin{equation}
Q = \log \tr{\omega}{(\pi^* \omega_B}) - A \varphi 
\end{equation}
for $A$  sufficiently large  and using the fact that $\varphi$ is uniformly bounded.\end{proof}

We will make use of this in the  key estimate:

\begin{lemma}\label{lem:bdd}
There exists a constant $C>0$ such that for $t$ in $[0,T)$,
\begin{equation}
\emph{tr}_{\omega_0}\, \omega(t) \le C.
\end{equation}
\end{lemma}
\begin{proof}
For any line bundle $L$, $\mathbb{P}(E) =
\mathbb{P}(E\otimes L)$.  Hence by replacing $E$ by $E\otimes A^{-1}$ for some
sufficiently ample line bundle $A$, we can assume that the dual bundle $E^*$ 
is generated by global sections. 

Fix an arbitrary $p \in B$. Choose $r$ sections $s_1,\ldots,
s_r$ of the dual
bundle $E^*$, which are linearly independent at $p$. Write
\[ f = s_1\wedge s_2\wedge\ldots \wedge s_r \]
for the corresponding section of the line bundle $\bigwedge^r E^*$. Let
$U\subset B$ be the set where $f$ does not vanish. On this set
$s_1,\ldots,s_r$ give a biholomorphism $\Phi: \pi^{-1}(U) \to U\times
\mathbb{P}^{r-1}$ such that the diagram
\begin{equation}
\begin{diagram}\label{diag1}
\node{\pi^{-1}(U) } \arrow{se,b,}{\pi}  \arrow[2]{e,t,b,}{\Phi}
\node[2]{U\times \mathbb{P}^{r-1}} \arrow{sw,r}{\textrm{pr}_1} \\
\node[2]{U}
\end{diagram}
\end{equation}
commutes, where $\textrm{pr}_1$ is the projection map onto the first factor. 

Let $\omega_{\textrm{prod}} =   (\textrm{pr}_1)^*\omega_B +
(\textrm{pr}_2)^* \omega_{\textrm{FS}}$ be the product metric on $U
\times \mathbb{P}^{r-1}$, where $\omega_{\textrm{FS}}$ is the Fubini-Study
metric on $\mathbb{P}^{r-1}$.  Define $\omegas =
\Phi^*\omega_{\textrm{prod}}$.  Then $\omegas$ is a smooth K\"ahler
metric on $\pi^{-1}(U)$, which we can also think of as a singular metric
on $X$. 

Fix a Hermitian metric $h$ on $E$, and write $h$ for the induced metric
on $E^*$, $\bigwedge^r E^*$ as well. 
We claim that there is a constant $C$ such that
\begin{equation}\label{trace}
	\tr{\omegas}{\omega_0} \leq C|f|^{-2}_h. 
\end{equation} 
This can be seen as follows.
The metric $\omega_0$ is uniformly equivalent to the metric $\oloc$, say, we obtain by
locally, over a small ball $V\subset B$, 
taking (non-holomorphic) sections $\sigma_1,\ldots,\sigma_r$ of $E^*$
which are pointwise orthonormal with respect to $h$, and then using
these sections to pull back the product metric on
$V\times\mathbb{P}^{r-1}$. Thus to compare the metrics $\omegas$
and $\oloc$ along each fiber of $\mathbb{P}(E)$, we need to compare
the Fubini-Study metrics on a given projective space constructed using
two different Hermitian metrics. On a given fiber $\mathbb{P}(E_q)$ the
metric $\oloc$ is constructed using the Hermitian metric $h$, whereas
$\omegas$ is constructed using the metric on $E_q$ in which
$s_1,\ldots,s_n$ give an orthonormal basis of $E_q^*$. At each point $q$
let
us write $A$ for the linear map $E_q\to\mathbb{C}^r$
given by $s_1,\ldots,s_r$, and $\lambda_1\leq \lambda_2\leq\ldots\leq
\lambda_r$ for the eigenvalues of  $A^*A$, where
the adjoint is formed by using the
metric $h$ on $E_q$ and the standard metric on $\mathbb{C}^r$. 
Lemma~\ref{lem:fubinistudy} below then implies that
$\omega_{\mathrm{sing}}\geq 
C\lambda_1\lambda_2\lambda_r^{-2}\omega_0$. 
Each eigenvalue is bounded above uniformly since
$s_1,\ldots,s_r$ are defined over the whole of $B$, so
$\lambda_1\lambda_2$ is bounded below by the determinant. 
The determinant of $A^*A$ is given
by $|f|_h^2$, so it follows that $\omegas\geq C|f|_h^2 \omega_0$.

To prove the lemma we will apply the maximum principle to the quantity
\begin{equation}
H = \log \left( |f|^3_h
\tr{\omegas}{\omega(t)} \right),
\end{equation}
on the set $\pi^{-1}(U)$.  
By the claim above, $H$ tends to negative infinity along
$X\setminus\pi^{-1}(U)$
and hence a maximum must occur in $\pi^{-1}(U)$ at each fixed time. 
 
We recall the well-known evolution inequality \cite{C} (see also \cite{Y, A}) for $\omega =\omega(t)$ solving (\ref{krf}):
\begin{equation} \label{cao}
\left( \ddt{} - \Delta \right) \log \tr{\hat{\omega}}{\omega} \le - \frac{1}{\tr{\hat{\omega}}{\omega}} g^{k \ov{\ell}} \hat{R}_{k \ov{\ell}}^{\ \ \, \ov{j} i} g_{i \ov{j}}, 
\end{equation} 
 where $\hat{\omega}$ is any fixed K\"ahler metric on $X$, and $\hat{R}_{k \ov{\ell}}^{\ \ \, \ov{j} i}$ are its curvature components. We will apply this with $\hat{\omega} = \omegas$.
 
 First note that  the curvature tensor of $\op$ has a lower bound
\begin{align} \nonumber
R_{i \ov{j} k \ov{\ell}} (\op) & \geq (\textrm{pr}_2^*
g_{\textrm{FS}})_{i \ov{j}} (\textrm{pr}_2^* g_{\textrm{FS}})_{k
\ov{\ell}} + (\textrm{pr}_2^*
g_{\textrm{FS}})_{i \ov{\ell}} (\textrm{pr}_2^* g_{\textrm{FS}})_{k
\ov{j}} \\ & \mbox{} - c (\textrm{pr}_1^* g_{B})_{i \ov{j}} (\textrm{pr}_1^*
g_{B})_{k \ov{\ell}}   
\end{align}
for a constant $c$ depending only on  the curvature of $g_B$, where the inequality of  tensors is meant in the sense of Griffiths. 
It follows that
\begin{equation} \label{r}
g^{i \ov{j}} g_{k \ov{\ell}} (R_{i \ov{j}}^{ \ \ \, \ov{\ell} k}(\omegas))
\ge - C(\tr{\omegas}{\omega}) (\tr{\omega}{\pi^*\omega_B}),
\end{equation}
for a uniform constant $C$.
Then compute, using (\ref{cao}), (\ref{r}) and the fact that the curvature
of the metric $h$ on $\bigwedge^r E^*$ is bounded by some multiple of
$\pi^* \omega_B$, 
\begin{equation}
\left( \ddt{} - \Delta \right) H \le C \tr{\omega}{\pi^*\omega_B} \le C'.
\end{equation}
where we used Lemma \ref{schwarz} for the second inequality.

It follows from the maximum principle that $H$ is bounded from above.   This gives the bound
\begin{equation} \label{z}
	\tr{\omegas}{\omega} \le C|f|_h^{-3}.
\end{equation} 
It follows that for some open subset $U'\subset\subset U$ containing
$p$, we have a bound $\tr{\omega_0}{\omega} \le C$ on $\pi^{-1}(U')$, 
where $C$ depends on
$p$. We can repeat the argument at each point of $B$, and by compactness
of $B$, a finite number of the open sets $U'$ cover $B$. Taking the
largest of the corresponding constants $C$ we obtain the required
result.
\end{proof}

We have used the following lemma in the proof:
\begin{lemma}\label{lem:fubinistudy}
	Let $A:\mathbb{C}^r\to\mathbb{C}^r$ be an invertible
	linear map, and write
	$\psi:\mathbb{P}^{r-1}\to\mathbb{P}^{r-1}$ for the induced map
	of projective spaces. Let $\omega_{\emph{FS}}$ be the standard
	Fubini-Study metric on $\mathbb{P}^{r-1}$. Then 
	\[ \psi^*\omega_{\emph{FS}} \geq \frac{\lambda_1\lambda_2}{\lambda_r^2}
	\omega_{\emph{FS}},\]
	where $0 < \lambda_1\leq \lambda_2\leq\ldots \leq
	\lambda_r$ are the eigenvalues of the matrix $A^*A$. 
\end{lemma}
\begin{proof}
	Let $\xi$ be a tangent vector of type $T^{1,0}$ at $x \in \mathbb{P}^{r-1}$  with $|\xi|^2_{\ofs}=1$.  Replacing $A$ by $UA$ for $U$ a unitary transformation  we may assume that $x$ is the point  $[1:0:\ldots:0]$. Choose  holomorphic coordinates  $z^i = Z_{i+1}/Z_1$ for $i=1, \ldots, r-1$ where 
 $Z_1, \ldots, Z_r$ are the homogeneous coordinates on $\mathbb{P}^{r-1}$.    Applying another unitary transformation (fixing $[1:0: \ldots :0]$), we may assume without loss of generality that $\xi = \partial/\partial z^1$. 
	
	The metric $\psi^*\omega_{FS}$ is given by
	\[ \psi^*\omega_{FS} = \sqrt{-1}
	\partial\bar{\partial}\log \left( \sum_{i,j=1}^n a_{ij}Z_j \ov{Z_i} \right),\]
	where $a_{ij}$ are the entries of $A^*A$. At the point
	$[1:0:\ldots:0]$ we then have
	\[ \psi^*\omega_{FS} = \frac{\sum a_{ij}dZ_j \wedge 
	d\ov{Z_i}}{a_{11}|Z_1|^2} -
	 \frac{\sum a_{1j} a_{i1} dZ_j \wedge d\ov{Z_i}}{a_{11}^2 |Z_1|^2},\]
	so
	\[  | \xi|^2_{\psi^*\ofs}  = \frac{a_{11}a_{22} -
	a_{12}a_{21}}{a_{11}^2} \geq
	\frac{\lambda_1\lambda_2}{\lambda_r^2},\]
	since the $2\times 2$ minor $a_{11}a_{22}-a_{12}a_{21}$ is
	bounded below by the product of the two smallest eigenvalues.
	Since $\xi$ was an arbitrary unit vector of type $T^{1,0}$, the result follows. 
\end{proof}

Note that Lemma \ref{lem:bdd} shows in particular that the diameter
of each fiber is uniformly bounded with respect to the evolving metric
$\omega$.  In fact, we can show that the diameters tend to zero as $t
\rightarrow T$.  For $b \in B$, write $F_b$ for the fiber $\pi^{-1}(b)$.

\begin{lemma}\label{lem:diam}
There exists a constant $C$ such that for all $b \in B$ and $t$ in $[0,T)$,
$$\emph{diam}_{g(t)} F_b \le C (T-t)^{1/3},$$
where we are writing $g(t)$ for the evolving metric $g(t)$ restricted to $X_s$.
\end{lemma}
\begin{proof}  The argument is similar to that of Lemma 3.2 in \cite{SW2}, but we include the proof for the sake of completeness. Let $p,q$ be two points in the fiber $F_b$. Under an
	identification $F_b\cong \mathbb{P}^{r-1}$ we can choose a line
	$\mathbb{P}^1\subset F_b$ containing $p,q$.
First note that from (\ref{reference}) we have
\begin{equation} \label{decay}
	\int_{\mathbb{P}^1} \omega(t) = 
	\frac{1}{T} \int_{\mathbb{P}^1} (T-t) \omega_0 \le C (T-t).
\end{equation}
Fix a background K\"ahler metric $\hat{g}$ on $\mathbb{P}^1$ and write
$g$ for the metric $g(t)$ restricted to $\mathbb{P}^1$.  Then
(\ref{decay}) gives us
\begin{equation} \label{fiber1}
\int_{\mathbb{P}^1} (\tr{\hat{g}}{g}) \, \hat{\omega} \le C(T-t),
\end{equation}
while from Lemma \ref{lem:bdd} we have
\begin{equation} \label{metricbound}
\tr{\hat{g}} g \le C.
\end{equation}
We wish to show that
$$d_{g} (p,q) \le C(T-t)^{1/3}.$$
Let $\ve = (T-t)^{1/3}$.   
Without loss of generality, we may assume that $p$ and $q$ lie in a fixed coordinate chart $U$ with holomorphic coordinate $z=x+iy$.  Moreover, we may assume that $p$ corresponds to the origin in $\mathbb{C}$,  and $q$ to the point $(x_0,0)$ with $x_0>0$.  We also assume  that the rectangle
$\mathcal{R} = [0, x_0] \times (- \ve, \ve) \subset \mathbb{R}^2 = \mathbb{C}$
is contained in  $U$.  Since   the fixed metric $\hat{g}$ is uniformly equivalent to the Euclidean metric in $\mathcal{R}$, we have  from (\ref{fiber1}),
$$\int_{-\ve}^{\ve} \left( \int_{0}^{x_0} (\tr{\hat{g}}{g}) dx \right) dy  = \int_{\mathcal{R}} (\tr{\hat{g}}{g}) dxdy \le  C(T-t).$$
It follows that for some $y' \in (-\ve, \ve)$,
$$\int_0^{x_0} (\tr{\hat{g}}g)(x, y') dx \le \frac{C}{\ve} (T-t) = C (T-t)^{2/3}.$$
If $p'$ and $q'$ are the points represented by $(0,y')$ and $(x_0, y')$ then 
\begin{eqnarray*}
d_g(p',q') & \le & \int_0^{x_0} \left(\sqrt{ g( \partial_x, \partial_x)} \right)(x,y') dx \\
& = & \int_0^{x_0} \left(\sqrt{ \tr{\hat{g}}{g}} \sqrt{ \hat{g}( \partial_x, \partial_x)} \right)(x,y') dx\\
& \le & \left( \int_0^{x_0} (\tr{\hat{g}}{g})(x,y') dx\right)^{1/2} \left( \int_0^{x_0} \left( \hat{g} ( \partial_x, \partial_x) \right)(x,y')dx \right)^{1/2} \\
& \le & C(T-t)^{1/3}.
\end{eqnarray*}
But from (\ref{metricbound}), we have
$$d_g(p,p') \le C d_{\hat{g}}(p,p') \le C' \ve = C'(T-t)^{1/3},$$
and similarly for $d_g(q,q')$.
Then
$$d_g(p,q) \le d_g(p,p') + d_g(p',q')  + d_g(q',q) \le C (T-t)^{1/3},$$
as required.
\end{proof}

It is natural to expect that in fact $\mathrm{diam}_{g(t)}F_b$
decays at the faster rate of
$(T-t)^{1/2}$, but we will not need this for our
application.

\section{Gromov-Hausdorff convergence}\label{sec:mainproof}

Using the estimates from Section \ref{sec:estimates} we now prove Theorem \ref{thm:main}.

\begin{lemma}\label{lem:dist}
	Write $d_t : X\times X\to \mathbb{R}$ for the distance function
	induced by the metric $\omega(t)$. There exists a sequence of
	times $t_i\to T$, such that the functions $d_{t_i}$ converge
	uniformly to a  function $d_\infty: X\times X\to\mathbb{R}$. 
	
	Moreover if 
	for $p,q\in B$
	we let $d_{B,\infty}(p,q)=d_\infty(\tilde{p},\tilde{q})$, where
	$\tilde{p}\in F_p$ and $\tilde{q}\in F_q$, then $d_{B,\infty}$ 
	defines a
	distance function on $B$, which is uniformly equivalent to that
	induced by $\omega_B$. 
\end{lemma}
\begin{proof}
	First note that the functions $d_t:X\times X\to\mathbb{R}$ are
	uniformly bounded. Indeed by Lemma \ref{lem:bdd} we have a
	constant $C$ such that
	$d_t(x,y) < \sqrt{C}d_0(x,y) < C'$ for any $t < T$ and $x,y\in X$. 
	In addition the functions $d_t:X\times X\to \mathbb{R}$
	are equicontinuous with respect to the metric on $X\times X$
	induced by $d_0$, since for $x,x',y,y'\in X$ we have
	\[ \begin{aligned}
		|d_t(x,y) - d_t(x',y')| &\leq |d_t(x,y) - d_t(x,y')| +
		|d_t(x,y') - d_t(x',y')| \\
		&\leq d_t(y,y') + d_t(x,x') \\
		&\leq \sqrt{C}(d_0(y,y') + d_0(x,x')).
	\end{aligned}\]
	By the Arzela-Ascoli theorem there is a sequence of times 
	$t_i\to T$ such that the functions $d_{t_i}$ converge uniformly
	to a continuous function $d_{\infty}:X\times X\to\mathbb{R}$. It
	follows that $d_\infty$ is non-negative, symmetric and 
	satisfies the triangle inequality. 

	Let $d_B: B\times B\to\mathbb{R}$ be the distance function on
	$B$ induced by the metric $\omega_B$. 
	From Lemma~\ref{schwarz} we have a constant $c > 0$ such that
	$d_t(x,y) \geq \sqrt{c} \, d_B(\pi(x),\pi(y))$. It follows that the
	limit $d_\infty$ satisfies 
	\begin{equation}\label{eq:lower}
		d_\infty(x,y)\geq \sqrt{c} \, d_B(\pi(x),\pi(y)).
	\end{equation}
	At the same time there is a constant $C_1$ such that for any
	$p,q\in B$ we have $d_0(F_p,F_q) < C_1d_B(p,q)$, so using both
	Lemma~\ref{lem:bdd} and Lemma~\ref{lem:diam} we have
	\[\begin{aligned}
		d_t(x,y)&\leq \mathrm{diam}_{g(t)}F_{\pi(x)} +
		\mathrm{diam}_{g(t)}F_{\pi(y)} + \sqrt{C}d_0(F_{\pi(x)},
		F_{\pi(y)}) \\
		&\leq 2C(T-t)^{1/3} + C_1\sqrt{C}d_B(\pi(x),\pi(y)).
	\end{aligned}\]
	This implies that
	\begin{equation}\label{eq:upper}
		d_\infty(x,y) \leq C_2d_B(\pi(x),\pi(y)).
	\end{equation}
	For $p,q\in B$ we now define $d_{B,\infty}(p,q) =
	d_\infty(\tilde{p},\tilde{q})$, where $\tilde{p}\in F_p$ and
	$\tilde{q}\in F_q$. This is independent of the choice of lifts
	$\tilde{p},\tilde{q}$ since if say $\tilde{p}'$ is a different
	lift of $p$, then by \eqref{eq:upper} and the triangle
	inequality we have
	\[ d_\infty(\tilde{p}',\tilde{q}) \leq
	d_\infty(\tilde{p},\tilde{q}) + d_\infty(\tilde{p},\tilde{p}') =
	d_\infty(\tilde{p},\tilde{q}),\]
	and by switching $\tilde{p}, \tilde{p}'$
	we get the reverse inequality. Moreover it follows from
	\eqref{eq:lower} and $\eqref{eq:upper}$ that $d_{B,\infty}$ is
	uniformly equivalent to $d_B$. This concludes the proof.
\end{proof}

\begin{theorem}
	In the notation of Lemma~\ref{lem:dist} we have
	$(X,d_{t_i})\to (B,d_{B,\infty})$ in the Gromov-Hausdorff sense,
	where we recall that $d_{t_i}$ is the distance function induced by
	the metric $\omega(t_i)$. 
\end{theorem}
\begin{proof}
	We use the characterization of Gromov-Hausdorff convergence
	given in, for example, \cite{F}. Given metric spaces
	$(X,d_X)$ and $(Y,d_Y)$, the Gromov-Hausdorff distance
	$d_{\textrm{GH}}(X,Y)$ is the infimum of all $\ve > 0$ such that the
	following holds. There exist maps $F:X\to Y$ and $G:Y\to X$ such
	that
	\[ \begin{aligned}
		|d_X(x_1,x_2) - d_Y(F(x_1),F(x_2))| & <
		\ve,\quad\text{ for all }x_1,x_2\in X,\\
		d_X(x, G\circ F(x)) & < \ve,\quad\text{ for all }x\in
		X,
	\end{aligned}\]
	and the two symmetric properties for $Y$ also hold. 

	Now define $F:X\to B$ to be the projection $F = \pi$, and $G:
	B\to X$ to be any (possibly discontinuous) map such that $F\circ
	G$ is the identity on $B$. Then for any $x_1,x_2,x\in X$ we have
	\[ \begin{aligned}
		|d_{t_i}(x_1,x_2) - d_{B,\infty}(F(x_1),F(x_2))| &=
		|d_{t_i}(x_1,x_2) - d_{\infty}(x_1,x_2)| \\
		d_{t_i}(x, G\circ F(x)) &\leq \mathrm{diam}_{g(t_i)}
		F_{\pi(x)},
	\end{aligned}\]
	and both of these quantities converge to zero uniformly in
	$x_1,x_2$ as $t_i\to T$, because $d_{t_i}\to d_\infty$ uniformly,
	and because of Lemma~\ref{lem:diam}.
	At the same time for any $p_1,p_2,p\in B$ we have
	\[ \begin{aligned}
		|d_{B,\infty}(p_1,p_2) - d_{t_i}(G(p_1),G(p_2))| &=
		|d_\infty(G(p_1),G(p_2)) - d_{t_i}(G(p_1),G(p_2))| \\
		d_{B_\infty}(p, F\circ G(p)) &= 0,
	\end{aligned}\]
	and again 
	the first quantity goes to zero as $t_i\to T$ since $d_{t_i}\to
	d_{\infty}$ uniformly. 

	This shows that as $i\to\infty$ we have $d_{\textrm{GH}}\left( (X,d_{t_i}),
	(B,d_{B,\infty})\right)\to 0$, which is what we wanted to prove.
\end{proof}

\begin{remark}
	Throughout the paper we have focused on manifolds of the form
	$\mathbb{P}(E)$, where $E$ is a vector bundle over a projective
	manifold $B$. It is likely that the
	behavior of the K\"ahler-Ricci flow 
	is similar for more general fibrations $X\to B$ over a
	K\"ahler manifold $B$, where the K\"ahler class of
	the metric at time $t$ tends to the pull-back of a K\"ahler
	class from $B$ as $t\to T$, just as in Equation
	\eqref{condition}. The main difficulty in extending Theorem
	\ref{thm:main} in this way is generalizing Lemma~\ref{lem:bdd}.
	Indeed it is crucial in the maximum principle argument that the
	fibers admit a metric with
	non-negative bisectional curvature. On the other it is likely
	that the arguments do extend without any additional difficulties
	for fibrations where the fibers admit such a metric.
\end{remark}

$^{*}$ Department of Mathematics \\
Rutgers University, Piscataway, NJ 08854\\

$^{**}$ Department of Mathematics \\ Columbia University, NY 10027 \\

$^{***}$ Department of Mathematics \\
University of California San Diego, La Jolla, CA 92093

\end{document}